\documentclass[12pt]{amsart}

\usepackage{times}
\usepackage[english]{babel}

\usepackage{enumitem}% http://ctan.org/pkg/enumitem
\usepackage{xcolor}

\usepackage{amsmath}
\usepackage{graphicx}
\usepackage{subcaption}
\usepackage{float}
\usepackage{amsthm}
\usepackage{amssymb}

\usepackage{mathtools}

\newtheorem{theorem}{Theorem}[section]

\newtheorem{proposition}[theorem]{Proposition}

\newtheorem{definition}[theorem]{Definition}

\theoremstyle{definition}
\newtheorem{remark}[theorem]{Remark}

\newcommand{\beb}{\begin{block}}
\newcommand{\enb}{\end{block}}

\newcommand{\bigzero}{\mbox{\normalfont\Large 0}}

\renewcommand{\a}{\alpha}

\newcommand{\w}{\omega}

\renewcommand{\l}{\left}
\renewcommand{\r}{\right}

\newcommand{\pr}{\partial}
\newcommand{\dm}{{\rm Dom}\, }

\newcommand{\ds}{\displaystyle}

\newcommand{\rom}[1]{\expandafter{\romannumeral #1\relax}}
\newcommand{\RNum}[1]{\uppercase\expandafter{\romannumeral #1\relax}}

\makeatother
\title[New characterization of the Hardy space and of other Hilbert spaces]{A new characterization of the Hardy space and of other Hilbert spaces of analytic functions}

\author{Natanael Alpay}
\address{Schmid College of Science and Technology\\
Chapman University\\
One University Drive\\
Orange, California 92866\\
USA}
\email{nalpay@chapman.edu}

\begin{document}
		\begin{abstract}
The Fock space can be characterized (up to a positive multiplicative factor)
as the only Hilbert space of entire functions in which the adjoint of derivation is multiplication by the complex variable. Similarly (and still up to a positive multiplicative factor)
the Hardy space is the only space of functions analytic in the open unit disk
for which the adjoint of the backward shift operator is the multiplication operator. In the present paper we characterize the Hardy space and some related reproducing kernel Hilbert spaces
in terms of the adjoint of the differentiation operator. We use reproducing kernel methods, which seem to also give a new characterization of the Fock space.
	\end{abstract}
\maketitle
\tableofcontents


\section{Introduction}
\setcounter{equation}{0}
The Fock (or Bargmann-Fock-Segal) space consists of the entire functions $f$ such that
\begin{equation}
  \label{gauss123}
\frac{1}{\pi}\iint_{\mathbb C}|f(z)|^2e^{-|z|^2}dxdy<\infty,
\end{equation}
and 
is the reproducing kernel Hilbert space with reproducing kernel
\begin{equation}
  \label{fock}
  e^{z\overline{\w}}.
\end{equation}
It is (up to a positive multiplicative factor) the unique Hilbert space of entire functions
in which 
\begin{equation}
\label{fock-equation}
  \pr_z^*=M_z,
\end{equation}
where $\pr_z$ denote the derivative with respect to $z$, 
and will be used throughout the work along with
the notation $(\pr_z f)(z)=f'(z)$. Furthermore, in \eqref{fock-equation} $M_z$ stands for multiplication by the variable $z$, e.g., $(M_z f)(z) = zf(z)$. We refer to the work of Bargmann
\cite{MR0157250,bargmann} for this result.
Formula \eqref{fock-equation} suggests to find similar characterizations for other important spaces of analytic functions. In particular, we have in mind the following spaces of functions analytic in
the open unit disk $\mathbb D$:\\

$(1)$ The Bergman space, which consists of the functions analytic in $\mathbb D$ and such that:
$$\ds\frac{1}{\pi}\iint_{\mathbb D}|f(z)|^2dxdy<\infty,$$
with reproducing kernel $\displaystyle \dfrac{1}{(1-z\overline{\w})^2} = \sum_{n=0}^{\infty} (n+1)z^n{\overline{\w}}^n$.\\

$(2)$ The Hardy space $\mathbf H^2$, when the condition is:
$$\lim_{r\rightarrow 1}\frac{1}{2\pi}\int_0^{2\pi}|f(re^{it})|^2dt<\infty,$$
with the reproducing kernel  $\displaystyle \dfrac{1}{1-z\overline{\w}} = \sum_{n=0}^{\infty} z^n\overline{\w}^n$.\\

$(3)$ The Dirichlet space, for which the functions vanish at the origin and satisfy
$$\frac{1}{\pi}\iint_{\mathbb D} |f^\prime(z)|^2dxdy<\infty,$$
with reproducing kernel $\displaystyle  -\ln (1-z\overline{\w}) = \sum_{n=1}^{\infty} \frac{z^n\overline{\w}^n}{n}$.\\

In the present work we approach this problem using reproducing kernel Hilbert spaces methods. We prove te following results.

\begin{theorem}
\label{thmhardy}
The Hardy space is, up to a positive multiplicative factor, the only reproducing kernel Hilbert space of functions analytic in $\mathbb D$, in which the equality
\begin{equation}
  \partial_z^* =M_z\partial_z M_z
  \label{hardy-equal}
\end{equation}
holds on the linear span of the kernel functions.  
\end{theorem}

Note that both in this, and in the next theorem, one could assume that the functions are analytic only in a neighborhood of the origin, and then use analytic continuation. We also note that
the unbounded operator $M_z\partial_z$ is diagonal, and acts on the polynomials as the number operator of quantum mechanics:
\[
M_z\partial_z(z^n)=nz^n,\quad n=0,1,\ldots,
\]
see e.g. \cite[p. 548]{fayngold2013quantum} which s the radial derivative for mathematics.\\

As mentioned above, the Hardy space of the open unit disk $\mathbb D$ has reproducing kernel $\frac{1}{1-z\overline{\w}}$. More generally, for every $\alpha\in(0,\infty)$, the function
$\frac{1}{(1-z\overline{\w})^\alpha}$ is positive definite in $\mathbb D$, as can be seen from the power series expansion of the function $\frac{1}{(1-t)^\alpha}$ with center at the origin as
\begin{equation}
	\label{eq:alpha_expension}
	\frac{1}{(1-z\overline{\w})^\alpha}=1+\sum_{n=1}^\infty\frac{\alpha(\alpha+1)\cdots (\alpha+n-1)}{n!}z^n\overline{\w}^n,\quad z,\w\in\mathbb D.
\end{equation}

We will
use a similar notation to Bargmann (see \cite[Remark 2g, page 203]{MR0157250}), and 
denote $\mathfrak H_\alpha$ to be the associated reproducing kernel Hilbert space, characterized by the following result.

\begin{theorem}
\label{general}
Let $\alpha>0$. Then the space $\mathfrak H_\alpha$ is, up to a multiplicative factor, the only reproducing kernel Hilbert space of functions analytic in $\mathbb D$,
in which the equality
	\begin{equation}
	\partial_z^* =M_z\partial_z M_z - (1-\alpha)M_z, \quad \alpha>0, %\in \mathbb N^+
	\label{general-eq}
	\end{equation}  
holds on the linear span of the kernel functions.  
\end{theorem}

The case $\alpha=1$ corresponds to the Hardy space and Theorem \ref{thmhardy}, and $\alpha=2$ corresponds to the Bergman space.
The case $\alpha=0$ would ``correspond'' to the Dirichlet space, in the sense that
\[
\lim_{\alpha\rightarrow 0}\frac{1}{\alpha}\left(\frac{1}{(1-z\overline{\w})^\alpha}-1\right)=-\ln(1-z\overline{\w}).
\]
Note that $\partial_z$ is not densely defined in the Dirichlet space (since $\pr_z k_\w$ is not in the Dirichlet space for $\w\not=0$), and therefore its adjoint is a relation and not an operator.
We were not able to get a counterpart of Theorem \ref{general} for $\alpha=0$, but we have the following result.

\begin{theorem}
\label{thmdir}
The Dirichlet space is, up to a positive multiplicative factor, the only reproducing kernel Hilbert space of functions analytic in $\mathbb D$, for which the equality
\begin{equation}
\pr_{z^2}^2k=\bar{\w}^2\pr_{z}\pr_{\bar\omega}k
  \label{dirich-equal}
\end{equation}  
holds for its kernel, pointwise for $z,\w\in\mathbb D$.
\end{theorem}
Note that \eqref{dirich-equal} is not an equality in the Dirichlet space, but rather, an equality between analytic functions. 
We give a similar characterization of the Fock space in Proposition \ref{proposition-fock}.\\

More generally, our analysis suggests a new direction in the study of the connections between reproducing kernel Hilbert spaces and operator models.
In particular, the following question is of interest: For which polynomials of two variables $p(x,y)$ does the equation
\[\pr_z^*=p(M_z,\pr) \]
characterize a reproducing kernel Hilbert space?

\begin{remark}
  {\rm When denoting inner products, we will sometimes mention explicitly the variable inside an inner
    product by writing $\langle f(z),g(z)\rangle$ rather than $\langle f,g\rangle$ to make the
    reading easier. See for instance equation \eqref{q1}.}
    \end{remark}  

    \begin{remark} {\rm A kernel $k(z,\w)$ analytic in $z$ and $\overline{\w}$ in a neighborhood
of $(0,0)$ (see Proposition \ref{propo-Hartog}) has a power series expansion at (0,0) of the form
\begin{equation}
  \label{knu}
k(z,\w)=\sum_{n,m=0}^\infty c_{n,m}z^n\bar{\w}^m.
\end{equation}
where
	\begin{equation}\label{eq:cnm}
		c_{n,m}=\langle z^n,\w^m\rangle^{-1}
	\end{equation}

	To ease the presentation, we associate to \eqref{knu}
    the infinite matrix $C(k)= (c_{m,n})_{n,m=0}^\infty$. Note that $C(k)$
    does not necessarily need to define a bounded operator in $\ell^2(\mathbb N_0)$.
	 For instance, for the Bergman kernel
	\[ \frac{1}{(1-z\bar{\w})^2} = 1 + 2 z\bar \w + 3 (z\bar\w)^2+\cdots, \]
	we have 
	\[
          C(k) =     
	\small
	\begin{pmatrix}
		1 &&&   \\
		& 2 && \bigzero\\
		&&    3 &     \\
		& \bigzero &&\ddots\\
	\end{pmatrix},
      \]
      which is unbounded on $\ell^2(\mathbb N_0)$.
	}
  \label{Ck}
\end{remark}

The paper consists of four sections besides the introduction. In Section 2 we review a number of definitions and results on reproducing kernel Hilbert spaces of analytic functions.
Sections 3, 4, and 5 contain proofs of Theorems \ref{thmhardy}, \ref{general}, and \ref{thmdir} respectively.

\section{Reproducing kernel Hilbert spaces}
\setcounter{equation}{0}
In this section we will briefly review the properties of reproducing kernel Hilbert spaces needed in the following sections. We first recall the definition

\begin{definition}
A reproducing kernel Hilbert space is a Hilbert space $(\mathcal H,\langle\cdot,\cdot\rangle)$ of functions defined in a non-empty set $\Omega$ such that there exists a complex-valued function $k(z,\omega)$ defined
on $\Omega\times\Omega$ and with the following properties:
	\begin{enumerate}
		\item $\forall \omega\in\Omega,\quad k_\omega: z\mapsto k(z,\omega)\in \mathcal{H}\to \mathcal{H}$
		\item $\forall f\in \mathcal H,\quad \langle f,k_\omega\rangle f(\omega).$
                \end{enumerate}
                \label{def21}
\end{definition}

The function $k(z,\w)$ is uniquely defined by the Riesz representation theorem, and is called the reproducing kernel of the space.
The reproducing kernel (kernel, for short) has a very important property: it is positive definite, that is, for all $N\in\mathbb N$, $\omega_1,\dots \omega_N\in\Omega$, and $c_1,\dots, c_N\in\mathbb C$, we have 
\[
\sum_{i,j=1}^{N}c_j\bar{c_i}k(\omega_i,\omega_j)\geq 0.
\]
In particular, it can be shown that the equation above implies that $k(z,\w)$ is Hermitian, i.e.
\begin{equation}
	\label{rtyui}
	k(z,\w)=\overline{k(\w,z)}.
\end{equation}

We refer to the book \cite{saitoh} for more information on reproducing kernel Hilbert spaces, and we recall that there is a one-to-one correspondence between positive definite functions on a given set
and reproducing kernel Hilbert spaces of functions defined on that set.
In the present work we are interested in the case where $\Omega$ is an open neighborhood of the origin, and where the kernels are analytic in $z$ and $\overline{\w}$.
The following result is a direct consequence of Hartog's theorem, and will be used in the sequel. For a different proof, see  \cite[p. 92]{donoghue}.

\begin{proposition}
	Let $\mathcal H$ be a reproducing kernel Hilbert space of functions analytic in $\Omega\subset\mathbb C$, with reproducing kernel $k(z,\w)$.
	Then the reproducing kernel is jointly analytic in $z$ and $\overline{\w}$.
	\label{propo-Hartog}
\end{proposition}

\begin{proof}
  Since the kernels belong to the space, we have that for every $\w\in\Omega$ the function $z\mapsto k(z,\w)$ is analytic in $\Omega$. From \eqref{rtyui} it follows that the kernel is also	analytic in $\overline{\w}$. Hartog's theorem (see \cite[p. 39]{mr1089193}) allows us to conclude that $k(z,\w)$ is jointly analytic in $z$ and $\overline{\w}$.
\end{proof}

When derivatives come into play, one then has \eqref{q2} below as the counterpart of \eqref{rtyui}:

\begin{proposition}
  \label{propo-1-2}
Under the hypothesis of the above discussion, the elements of the associated reproducing kernel Hilbert space are analytic in $\Omega$ and the following hold:
\begin{equation}
\label{q1}  
(\partial_w f)(\w)=\langle f(z),\partial_{\overline{\w}}k_{\w}(z)\rangle
\end{equation}  
and
\begin{equation}
\label{q2}
\pr_z k(z,\w_0)|_{z=z_0} = \overline{\pr_{\bar{\w}} k(\w_0,\w)|_{\w=z_0}  }.
\end{equation}  
\end{proposition}

\begin{proof}
The proof of \eqref{q1} can be found in  \cite[Theorem 9, p. 41]{MR1478165}.
We give the proof of \eqref{q2}, where as in Definition \ref{def21} and in the rest of the work, we use the notation: $k_\beta:z\mapsto k(z,\beta)$ where $\beta\in\Omega$.\smallskip

Setting $f(z)=k(z,\w_0)$ in \eqref{q1}  gives

\[
  \pr_z k(z,\w_0)|_{z=z_0} = \langle k(z,\w_0),\pr_{\bar\w}k(z,\w)|_{\w=z_0}\rangle 
\]
and so we have:
\[
\overline{\pr_z k(z,\w_0)|_{z=z_0}}=\langle \pr_{\bar\w}k(z,\w)|_{\w=z_0},k(z,\w_0)\rangle \pr_{\bar\w}k(z,\w)|_{z=\w_0,\w=z_0},
\]

and hence the result.
\end{proof}

For some special cases, the reader could also check \eqref{q2} for $k(z,\w)=f(z\bar{\w})$ or for $k(z,w)=a(z)\overline{a(w)}$, where $a(z)$ is analytic
in some open subset of the complex plane. 
In particular, for the latter example we have:
\[
  \pr_z k(z,\w_0)|_{z=z_0}=a^\prime(z_0)\overline{a(\w_0)}
\]
on the one hand, and
\[
\pr_{\bar{\w}} k(\w_0,\w)|_{\w=z_0}=a(\w_0)\overline{a^\prime(z_0)}
\]
on the other hand, and hence taking conjugates we see that \eqref{q2} holds. Since every positive definite function can be represented as
an infinite sum of functions of the form $a(z)\overline{a(w)}$ (this is Bergman's
reproducing kernel formula, see \cite{aron}), this would give another way to prove \eqref{q2}, after justifying interchange of sum and derivatives, but we preferred to give a direct proof.\\

The following is a main technical result that we will need in the proofs of the theorems.

\begin{proposition}
  \label{prop-2-4}
Let $k(z,\w)$ be positive definite and jointly analytic in $z$ and $\overline{\w}$ for $z,\w$ in an open subset $\Omega$ of the complex plane. Assume that the operator $\partial_z$ is densely defined in
the associated reproducing kernel Hilbert space $\mathcal H(k)$.
Then $\pr_z$  is closed and in particular has a densely defined adjoint $\pr^*_z$ which satisfies $\pr_z^{**}=\pr_z$.
\end{proposition}

\begin{proof}
	Let $(f_n)$ be a sequence of elements in  $\dm \pr$ and let $f,g\in \mathcal H$ be such that 
	\begin{align*}
	f_n\to f\\
	\pr f_n\to g		
	\end{align*}
	where the convergence is in the norm. Since weak convergence follows from strong convergence, using \eqref{q1}, we have for every $\w\in\Omega$ that
	\[\l\langle f_n,\pr_{\bar{\w}}k_{\w}\r\rangle \to \l\langle f,\pr_{\bar{\w}}k_\w\r\rangle \quad{\rm and}\quad \l\langle \pr f_n,k_\w\r\rangle \to \l\langle g,k_\w\r\rangle ,\]
	
	where the brackets denote the inner product in $\mathcal H(k)$.	Hence it follows that 
	 \[ \lim_{n\rightarrow\infty} f_n'(\w)=f'(\w)\quad{\rm and}\quad \lim_{n\rightarrow\infty} f_n'(\w)=g(\w).\]
	
	Thus $g=f'$, and hence $\pr$ is closed. Hence, $\pr$ has a densely defined adjoint and $\pr^{**}=\pr$; see e.g. \cite[Theorem VIII.1, pp. 252-253]{MR751959}.
\end{proof}  

As an application we prove the following characterization of the Fock space.
In the statement, one could assume the functions analytic only in a neighborhood of the origin, and then use analytic continuation. 

\begin{proposition}
\label{proposition-fock}  
The Fock space is the unique (up to a positive multiplicative factor) reproducing kernel Hilbert space of entire functions where the equation
\[
\partial_z^*=M_z
\]
holds on the linear span of the kernels (in particular the kernel functions are in the domain of $\pr^*$ and of $M_z$).
\end{proposition}

\begin{proof} Let $k(z,w)$ be the reproducing kernel of the space in the proposition. We want to show that $k(z,w)=ce^{z\overline{\w}}$ for some $c>0$.
	From Proposition \ref{propo-Hartog} the kernel is jointly analytic in $\mathbb D$. Since $\pr^*=M_z$, it follows that 
	\[\langle \pr_z^*k(z,\w), k(z,\nu)\rangle = \langle M_zk(z,\w),k(z,\nu)\rangle . \]
	
	Evaluating each side yields the following: For the right hand side we get
	\[ \langle M_zk(z,\w),k(z,\nu)\rangle = \nu k(\nu,\w) \]
	
	since $M_zk(z,\w)=zk(z,\w)$. The left hand side yields
\begin{equation}
		\label{eq:12345}
	\begin{aligned}
		\langle \pr_z^* k(z,\w),k(z,\nu)\rangle &= \langle k(z,\w),\pr_zk(z,\nu)\rangle \\
		&= \overline{\langle \pr_zk(z,\nu),k(z,\w)\rangle }\\
		&= \overline{\pr_z k(z,\nu)|_{z=\w}}\\
		&= \overline{\pr_\w k(\w,\nu)}\\
		&= \pr_{\bar{\w}}k(\nu,\w),
	\end{aligned}
\end{equation}
where we have used \eqref{q2} to go from the penultimate line to the last one.
Thus we obtain that $\pr_{\bar{\w}}k(\nu,\w)=\nu k(\nu,\w)$, which is a differential
equation with the solution
\[ k(\nu,\w)=c(\nu)e^{\nu\bar{\w}},
\]
where the function $c(\nu)$ is an entire function of $\nu$ (since $k(\nu,\w)$ and $e^{\nu\overline{\w}}$ are entire functions of $\nu$).
But $k(\nu,\w)=\overline{k(\w,\nu)}$. Hence $c(\nu)=\overline{c(\nu)}$ so that $c(\nu)$ is real valued. Using the Cauchy-Riemann equations, we see that $c(\nu)$ is
a constant, which is furthermore positive since the kernel is positive. 
\end{proof}

\begin{remark} The Fock space can be described in a geometric way by the Gaussian weight as in \eqref{gauss123}. The Gaussian weight has other characterizations. We mention in particular the one from information theory: the Gaussian distribution $\frac{1}{\sqrt{2\pi}}e^{-\frac{x^2}{2}}$
  maximizes the entropy
\[
  -\int_{\mathbb R}f(x)\ln f(x)dx
\]
among all probability distributions with zero mean and second moment equal to $1$;
  see e.g. \cite[Exercise 4, p. 50]{MR2363070} and \cite[Theorem 8.3.3, p. 240]{ash_book}. It can also be characterized (after normalization) as the unique continuous radial weight function $\w(z)=\frac{1}{\pi}e^{-|z|^2}$ such that for polynomial $p$ and $q$ under the inner product 
  $$ \langle p,q\rangle = \frac{1}{\pi}\iint_{\mathbb C}p(z)\bar q(z)\w(z)dA(z), $$
  the operator of multiplication and differentiation are adjoint to each other; see \cite{MR0157250} (and J. Tung's thesis \cite{Tung}). 
 \end{remark}

\section{Proof of Theorem \ref{thmhardy}}
\setcounter{equation}{0}
We first check that the kernel $k_\w(z)=\frac{1}{1-z\overline{\w}}$ is a solution of \eqref{hardy-equal}, i.e.
\[
  \left\langle \pr_zg,k(z,\omega)\right\rangle = \left\langle g, \pr_z^*k(z,\omega) \right\rangle=  \left\langle g, M_z\pr_zM_zk(z,\omega)\right\rangle,
\]
with $g(z)=\frac{1}{1-z\w^*}$.
To verify the above, we compute the left side of the equation and have
\begin{align*}
	\langle \pr_zk_\nu(z),k_\w(z)\rangle &= \l\langle \pr_z\l( \frac{1}{1-z\bar{\nu}} \r),k_\w(z)\r\rangle  =\l\langle \frac{\nu}{(1-z\nu)^2} , k_\w(z) \r\rangle =\frac{\bar{\nu}}{(1-\w\nu)^2}.
\end{align*}
Similarly, we independently calculate the right hand side as
\begin{align*}
	\langle k_{\bar{\nu}}(z),M_z\pr_zM_zk_\w\rangle &= \l\langle k_{\bar{\nu}}(z),  M_z\pr_z \l( \frac{z}{1-z\bar\w} \r)\r\rangle \\
	&= \l\langle k_{\bar{\nu}}(z),\frac{z}{(1-z\bar\w)^2}\r\rangle \\
	&= \overline{ \l\langle \frac{z}{(1-z\bar\w)^2},k_\nu(z) \r\rangle }\\
	&= \frac{\bar\nu}{(1-\w\bar\nu)^2},
\end{align*}
which comes to be the same as the left hand side.\\

To prove the converse we apply \eqref{hardy-equal} to kernels,
then we
use analyticity to find the kernel via its Taylor expansion at the origin. Let $\w,\nu\in \mathbb D$. From \eqref{hardy-equal} we get
\begin{equation}\label{eq1.2}
  \langle \partial_z k_\w,k_\nu\rangle=\langle k_\w,\pr_z^*k_\nu\rangle \langle 
k_\w,M_z\partial_z M_zk_\nu\rangle.
\end{equation}

We rewrite \eqref{hardy-equal} as 
	\begin{align*}
	\partial_z^*f=z(\partial_z zf)&=z(zf'+f)=z^2f'+zf.
	\end{align*}

By hypothesis the kernel functions belong to the domain of $\partial_z^*$ and we have $\partial_z^{**}=\partial_z$ by Proposition \ref{prop-2-4}.
Therefore,
      By by \eqref{eq:12345} we obtain
\begin{equation}
\label{3=4=5}
\langle \pr_z^* k_\w,k_\nu \rangle = (\pr_{\bar{\w}} k)(\nu,\w).
\end{equation}
Then, using the two end sides of \eqref{eq1.2}, we get
\begin{align*}
	\langle M_z\pr_z M_z k_\omega(z), k_\nu(z)\rangle 
 & = \overline{	\langle k_\omega(z),M_z\pr_z M_z k_\nu(z)\rangle }\\
 &=\overline{\langle k (z,\omega),z^2\pr_z k(z,\nu) + z k(z,\nu)\rangle }\\
  &=\overline{\langle k (z,\omega),z^2\pr_z k(z,\nu)\rangle }  + \overline{\langle k (z,\omega) ,  z k(z,\nu)\rangle }\\
  & = \bar{\w}^2\partial_{\bar\w}k(\nu,\w)   +\bar{\w} k(\nu,\w)
\end{align*}
where we have used \eqref{q2} to go from the penultimate line to the last one.
Considering $k=k(z,\w)$, we get the partial differential equation 
$$\pr_{\bar{\omega}} k =  z^2\pr_{z} k +z k,$$
where $k=k(z,\omega)$, or equivalently by replacing the role of $z$ and $\w$ we obtain
\begin{equation}\label{eq1.3}
\pr_z k =\bar{\omega}^2 \pr_{\bar{\omega}} k + \bar{\omega} k.
\end{equation}

The kernel is analytic in $z$ and $\overline{\w}$ near the origin, and hence can be written as
\eqref{knu}. So we can rewrite \eqref{eq1.3} as
\begin{align*}
	\sum_{n=1}^{\infty}\sum_{m=0}^{\infty} n c_{n,m}z^{n-1}\bar{\omega}^m=  \sum_{n=0}^{\infty}\sum_{m=1}^{\infty}mc_{n,m}z^n\bar{\omega}^{m+1}+ \sum_{n=0}^{\infty}\sum_{m=0}^{\infty}c_{n,m}z^n\bar{\omega}^{m+1},
\end{align*}
which can also be written as:
\begin{align*}
	&\sum_{n=0}^{\infty}(n+1)c_{n+1,0}z^n
	+\sum_{n=0}^{\infty}(n+1)c_{n+1,1}z^n\overline{\omega}
	+\sum_{n=0}^{\infty}\sum_{m=2}^{\infty} (n+1) c_{n+1,m}z^{n}\bar{\omega}^m\\
	&= \sum_{n=0}^{\infty} c_{n,0}z^n\overline{\omega}
	+\sum_{n=0}^{\infty}\sum_{m=2}^{\infty} mc_{n,m-1}z^{n}\overline{\w}^{m}.\\
\end{align*}

Now we compare the terms on two sides. First we look at the part which is constant with respect to $\w$ and get
$\ds\sum_{n=0}^{\infty}(n+1)c_{n+1,0} z^n=0$. Hence 
\begin{equation}
  c_{n+1,0}=0,
  \label{cn1}
  \end{equation}
for all $n\in \mathbb N_0$.

Consider the  coefficients of $z^n\overline{\w}$ on both sides. Then we have $\ds\sum_{n=0}^{\infty}(n+1)c_{n+1,1} z^n\overline{\w}=\sum_{n=0}^{\infty} c_{n,0}z^n\overline{\omega} $. 
Hence 
\begin{equation}\label{eq:1234567}
	(n+1)c_{n+1,1}=c_{n,0},
\end{equation}
for all $n\in\mathbb N_0$. Note that for $n=0$ we get $c_{0,0}=c_{1,1}$.

Consider the terms $z^n\overline{\w}^m$, $m\geq2$. Then $\ds 	\sum_{n=0}^{\infty}\sum_{m=2}^{\infty} mc_{n,m-1}z^{n}\overline{\w}^{m} = \sum_{n=0}^{\infty}\sum_{m=2}^{\infty} (n+1) c_{n+1,m}z^{n}\bar{\omega}^m$.
Hence 
\begin{align}\label{eq1.6}
	mc_{n,m-1}  =(n+1)c_{n+1,m},
\end{align}
for all $n\in \mathbb{N}_0$ and $m=2,3,...$. Note if $m=n+1$, then $(n+1)c_{n+1,n+1}=(n+1)c_{n,n}$. So 
\begin{equation}
	\label{hardy678}
	c_{0,0}=c_{1,1}=c_{2,2}=\dots .
\end{equation}

We now check that $c_{n,m}=0$ when $n\neq m$. For $0<m< n+1$ using \eqref{eq:1234567} and \eqref{eq1.6} it follows that
\[c_{n+1,m}=\alpha_{n,m} c_{n+1-m,0},\]
where $\alpha_{n,m} = \frac{m}{n+1}\frac{m-1}{n}\cdots\frac{1}{n+2-m}\neq0$, then $c_{n+1,m}=0$ by \eqref{cn1} for $n+1>m$. The case $m>n$ is obtained by symmetry.\\

Hence, all off-diagonal entries of the matrix $C(k)$ (defined in Remark
\ref{Ck} will be zero, and it follows from \eqref{hardy678} that
$k(z,\w)=\frac{c_{0,0}}{1-z\overline{\w}}$. This ends the proof of the theorem.\qed\\

If we assume that the powers of $z$ are in the domain of $\partial^*$ and of $M_z$ one has a simpler proof for the characterization given in Theorem \ref{thmhardy} of the Hardy space,
close in spirit to Bargmann's arguments. We note that conditions $(1)$-$(4)$ in the statement of the next result are satisfied by $\mathbf H^2$,

\begin{proposition}\label{propo123}
	Let $\mathcal H$ be a reproducing kernel Hilbert space of functions analytic in a neighborhood of the origin and such that: 
	\begin{enumerate}
		\item $M_z$ bounded,
		\item $\{z^n\}_{n=0}^\infty\subset \dm \pr$,
		\item $\dm \pr \subset \dm \pr^*$,
		\item $\pr^*=M_z\pr M_z$.
	\end{enumerate}
	Then $\mathcal H=\mathbf H^2$.% up to a scalar.
\end{proposition}
\begin{proof}
	Let the kernel $K$ of $\mathcal{H}$ have the form in \eqref{knu}.
	From Proposition \ref{propo-Hartog} the kernel is jointly analytic in $\mathbb D$.
  Take $f(z)=z^n$ and $g(z)=z^m$, then
	\begin{align*}
	\langle f ,\pr g\rangle &= \langle z^n,mz^{m-1}\rangle &	\langle \pr^* f,g\rangle &=\langle z^2 f'+zf,g\rangle \\
	&=m\langle z^n,z^{m-1}\rangle & 	&= \langle n z^{n+1}+z^{n+1},z^m\rangle \\
	& & 	&=(n+1) \langle  z^{n+1},z^m\rangle .
	\end{align*}
	Since $	\langle f,\pr g\rangle \langle \pr^* f,g\rangle $, we obtain
	\begin{equation}
	(n+1)\langle z^{n+1},z^m\rangle m \langle z^n,z^{m-1}\rangle .
	\label{prop-eq-1}
	\end{equation}
	For $m=n+1$, we have 
	\begin{align*}
	(n+1)\langle z^n,z^n\rangle = (n+1)\langle z^{n+1},z^{n+1}\rangle \implies \langle z^n,z^n\rangle = \langle z^{n+1},z^{n+1} \rangle ,
	\end{align*}
	thus the diagonal entries are nonzero.
	Now we are left to show that if $n\neq m$, $\langle z^n,z^m\rangle =0$. From \eqref{prop-eq-1} we get 
	\begin{equation}
			\langle z^{n+1},z^m\rangle = \frac{m}{n+1}\langle z^n,z^{m-1}\rangle .
			\label{prop-eq-2}
	\end{equation}
		
	Take $f(z)=z^n,n\neq 0$, and $g(z)\equiv1$; then
	\begin{align*}
	\langle f, \pr g \rangle =\langle \pr^* f ,g \rangle &= \langle z^2f'+zf,g\rangle \\
	&=\langle n z^{n+1}+z^{n+1},1\rangle \\
	&=(n+1)\langle z^{n+1},1\rangle .
	\end{align*}
	However $\langle f, \pr g \rangle = 0$, hence $\langle z^{n+1},1\rangle 0$, which also gives $\langle 1,z^{m+1}\rangle 0$. Then from \eqref{eq:cnm} and \eqref{prop-eq-2} all the off-diagonal coefficients $c_{n,m}$ are equal to $0$.
\end{proof}

More generally, with the same hypothesis as in Proposition \ref{propo123}, one could replace $M_z\partial_z$ by a (possibly unbounded) diagonal operator defined as follows:
\[
D(z^n)=\alpha_n z^n,\quad n=0,1,2,\ldots,
\]
with $\alpha_n>0$ for $n\ge 1$ and $\alpha_0$ arbitrary. Such $D$ is called a radial differential operator in the literature. Then we get
\[
\langle z^n,z^m\rangle=\delta_{n,m}\frac{n!}{\alpha_n\cdots \alpha_1}\langle 1,1\rangle.
\]  
Taking $\beta^{-1}=\langle 1,1\rangle $, and using \eqref{eq:cnm}, the reproducing kernel is given by
\[
k(z,\w)=\beta\sum_{n=0}^\infty\frac{\alpha_n\cdots \alpha_1}{n!}z^n\overline{\w}^n
\]
by \eqref{eq:cnm}, provided the radius of convergence of the above series is strictly positive.

\section{Proof of Theorem \ref{general}}
\setcounter{equation}{0}

To prove Theorem \ref{general}, we use the same strategy as in the previous section. The kernel $\frac{1}{(1-z\overline{\w})^\alpha}$ is a solution of $\pr^*=M_z\partial_z M_z -(1-\alpha)M_z$, this applied to this kernel gives us 
\begin{align*}
	\partial^* k(z,\w) &=\left(M_z\partial_z M_z -(1-\alpha)M_z\right) \left(\frac{1}{(1-z\bar{\w})^{\alpha}}\right)\\
	&=\frac{z}{(1-z\bar\omega)^\alpha}+\alpha \frac{z^2\bar\omega}{(1-z\bar\omega)^{\alpha+1}}-(1-\alpha)\frac{z}{(1-z\bar\omega)^\alpha}\\
	&=\frac{\alpha z}{(1-z\bar\omega)^{\alpha+1}}\\
	&=\partial_{\bar{\w}}\left(\frac{1}{(1-z\bar{\w})^\alpha}\right)\\
	&=\partial_{\bar{\w}}k(z,\w).
\end{align*}
which implies 
\[ \langle \partial^*k_\nu(z), k_\w(z)\rangle = \langle k_\nu(z),  \partial_{\bar\w} k_\w(z) \rangle .  \]
Additionally, we get the relation
$
	z(1-z\bar\omega) + \alpha z^2\bar\omega -(1-\alpha) z (1-z\bar\omega)=\alpha z.
$\\

As we see again, indeed for $\alpha=1$ we have the Hardy case. To prove the converse we apply \eqref{general-eq} to kernels, and find a partial differential equation satisfied by the reproducing kernel. Then we
use analyticity to find the kernel via its Taylor expansion at the origin. Let $\w,\nu\in \mathbb D$, then from \eqref{general-eq} we get
\begin{equation}\label{general_one_side}
\langle \partial k_\w,k_\nu\rangle=\langle k_\w,\pr^*k_\nu\rangle \langle 
k_\w,M_z\partial M_zk_\nu +(\alpha-1)M_zk_\nu\rangle.
\end{equation}

We rewrite \eqref{general-eq} as 
\begin{equation}
\begin{split}
\partial^*f=z(\partial zf)+(\alpha-1)zf&=z^2f'+zf+ \alpha z f - zf\\
&=z^2f'+  \alpha zf .
\end{split}
\label{general-eq-v2}
\end{equation}

From the calculation above similar to \eqref{eq:12345}, it follows that 
$\langle \pr_z k(z,w),k(z,\nu) \rangle = \pr_z k(\nu,\w)$, thus from \eqref{general-eq-v2} and the two end sides of \eqref{general_one_side}.
Equation 
\eqref{eq:12345}

still holds here (it is a general computation valid for kernels analytic in $z$ and $\w$) and we have
\begin{align*}
\partial_z k(\nu,\w)&=\pr_z k(z,\w)|_{z=\nu}\\
&=\langle \pr_zk_\w,k_\nu \rangle \\
&=\langle k_\w,\pr_z^*k_\nu \rangle \\
&= \langle k_\omega,M_z\partial M_zk_\nu -(\alpha-1)M_zk_\nu\rangle \\
&=\langle k_\omega,\nu^2\pr_z k_\nu+\a\nu k_\nu\rangle \\
&=\overline{\langle \nu^2 \pr k_\nu +  \a\nu k_\nu , k_\omega\rangle  }\\
&=\bar{\omega}^2\pr k(\nu,\omega) +  \a\bar{\omega} k(\nu,\omega).
\end{align*}

Thus we get the partial differential equation
\begin{equation}\label{general-diff-eq}
\pr_z k =\bar{\omega}^2 \pr_{\bar{\omega}} k +\alpha \bar{\omega} k.
\end{equation}

The kernel is analytic in $z$ and $\overline{\w}$ near the origin, and hence can be written as
\[
k(\nu,w)=\sum_{n,m=0}^\infty c_{n,m}\nu^n\bar{\w}^m.
\]

So we can rewrite \eqref{general-diff-eq} as
	\begin{align*}
	\sum_{n=1}^{\infty}\sum_{m=0}^{\infty} n c_{n,m}\nu^{n-1}\bar{\omega}^m=  \sum_{n=0}^{\infty}\sum_{m=1}^{\infty}mc_{n,m}\nu^n\bar{\omega}^{m+1}+\alpha \sum_{n=0}^{\infty}\sum_{m=0}^{\infty}c_{n,m}\nu^n\bar{\omega}^{m+1},
	\end{align*}
	which can also be written as:
\begin{align*}
	&\sum_{n=0}^{\infty} (n+1)c_{n+1,0}\nu^{n}  
	+\sum_{n=0}^{\infty} (n+1)c_{n+1,1}\nu^{n}\bar{\omega} 
	+\sum_{n=0}^{\infty}\sum_{m=2}^{\infty} (n+1)c_{n+1,m}\nu^{n}\bar{\omega}^m  \\
	&=\sum_{n=0}^{\infty}\alpha c_{n,0}\nu^n\bar{\omega}
	+\sum_{n=0}^{\infty}\sum_{m=2}^{\infty}(\alpha +(m-1))c_{n,m-1}\nu^n\bar{\omega}^{m}  .
\end{align*}

Now we can consider the following cases: First we compare the coefficients for the terms with constant $\bar\w$. Then we have: $\ds\sum_{n=0}^{\infty}(n+1)c_{n+1,0}\nu^n=0$. 
Hence 
\[c_{n+1,0}=0\]
for all $n\in\mathbb N_0$.\\

Consider the coefficients of $\nu^n\overline{\w}$. Then we have:$\ds\sum_{n=0}^{\infty}(n+1)c_{n+1,1}\nu^n\bar{\omega}=\sum_{n=0}^{\infty}\alpha c_{n,0}\nu^n\overline{\w}$.
Hence 
\[(n+1)c_{n+1,1}=\alpha c_{n,0},\]
for all $n\in \mathbb N_0$. Note that for $n=0$ we get $c_{0,0}=\alpha c_{1,1}$.\\

Consider the terms $\nu^n\overline{\w}^m$, $m\geq2$; then we have
$$\ds \sum_{n=0}^{\infty}\sum_{m=2}^{\infty} (n+1)c_{n+1,m}\nu^{n}\bar{\omega}^m  = \sum_{n=0}^{\infty}\sum_{m=2}^{\infty}(\alpha +(m-1))c_{n,m-1}\nu^n\bar{\omega}^{m}.$$
Hence
\begin{align}\label{general-eq-2}
(n+1)c_{n+1,m}=(m+\alpha-1) c_{n,m-1},
\end{align}
for all $n\in \mathbb N_0$. Note that if $m=n+1$, then $(n+1)c_{n+1,n+1}= (n+\alpha)c_{n,n}$. So
\[ c_{n,n}=\left(\frac{n+1}{n+\alpha}\right)c_{n+1,n+1}.\]
we see that the diagonal entries are equal (up to a constant) to the Taylor coefficients in  \eqref{eq:alpha_expension}.\smallskip

We now check that $c_{n,m}=0$ when $n\neq m$. For $0\leq m \leq n+1$, it follows from \eqref{general-eq-2} that 
\[c_{n+1,m}=\phi_{\alpha,n,m} c_{n+1-m,0},\]
for $\phi_{\alpha,n,m} = \frac{m+\alpha-1}{n+1}\frac{m+\alpha-2}{n}\cdots\frac{\alpha}{n+2-m} \neq0$, and hence the conclusion using \eqref{cn1}. The case $m>n$ follows by symmetry.
Hence from these cases and by symmetry, all off-diagonal entries of $C(k)$ will be zero, and this completes the proof. \qed

\section{Proof of Theorem \ref{thmdir}}
\setcounter{equation}{0}

While with similar spirit in proof structure, unlike in proofs for Theorems \ref{thmhardy} and \ref{fock}, we prove \eqref{dirich-equal} for the kernel pointwise for $z,\w\in\mathbb D$.
Let $k(\nu,\w)$ be a solution of \eqref{dirich-equal}, with power series expansion
\[
k(\nu,\w)=\sum_{n=0}^{\infty}\sum_{m=0}^\infty c_{n,m}\nu^n\bar{\w}^m.
\]
Since $k(0,0)=0$ by hypothesis, we have  $c_{0,0}=0$ (without the condition $k(0,0)=0$ any constant function is a solution of \eqref{dirich-equal}). We have
\begin{align*}
	\pr_\nu^2 k &= \sum_{n=2}^{\infty} \sum_{m=0}^{\infty}  c_{n,m} n(n-1)\nu^{n-2}\bar\omega^{m}\\ 
	\bar\omega^2 \pr_\nu \pr_{\bar \omega} k & = \sum_{n=1}^{\infty}\sum_{m=1}^{\infty}c_{n,m} n m \nu^{n-1}\bar\omega^{m+1}.
\end{align*}
So we can rewrite \eqref{dirich-equal} in terms of the power series expansion of kernel as:
\begin{equation}
	\sum_{n=2}^{\infty} \sum_{m=0}^{\infty}  c_{n,m} n(n-1)\nu^{n-2}\bar\omega^{m}  = \sum_{n=1}^{\infty}\sum_{m=1}^{\infty}c_{n,m} n m \nu^{n-1}\bar\omega^{m+1},
	\label{diri-eq-kernel}
\end{equation}

which is equivalent to 
\begin{equation}
	\sum_{n=2}^{\infty} \sum_{m=0}^{\infty}  c_{n,m} n(n-1)\nu^{n-2}\bar\omega^{m}  = 	
	\sum_{m=1}^{\infty}c_{1,m} m \bar\omega^{m+1}
	+\sum_{n=2}^{\infty}\sum_{m=1}^{\infty}c_{n,m} n m \nu^{n-1}\bar\omega^{m+1}.
	\label{diri-eq-kernel1}
\end{equation}

Comparing on both sides the part independent of $\nu$ we get
\begin{equation}
  \label{345}
  \sum_{m=1}^{\infty} c_{1,m}m \bar\omega^{m+1}=0,
\end{equation}
as we have no corresponding terms on the left side.\smallskip

Let $n=2$. Then
\begin{equation}
	\sum_{m=0}^{\infty}  c_{2,m} 2 \bar\omega^{m}  = \sum_{m=1}^{\infty}c_{2,m} 2 m \nu \bar\omega^{m+1}.
	\label{diri-eq-n=2}
\end{equation}

We make the change of index $M=m+1$ in \eqref{345}, and obtain
\begin{equation}
	\sum_{M=2}^{\infty} c_{1,m-1}(M-1) \bar\omega^{M}=0.
	\label{diri-eq-n=1}
\end{equation}
From equations \eqref{diri-eq-n=1} and \eqref{diri-eq-n=2}, it follows now that 
\[c_{2,0}=c_{2,1}=0\quad {\rm and }\quad 2c_{2,M}=(M-1)c_{1,M-1}~~{\rm for}~~ M>2.\]

Considering equation \eqref{diri-eq-kernel} and making the change of index $N=n-2$, $M=m$ to the right side, and $N=n-1$, $M=m+1$ to the left side, we get
\begin{equation}
	\sum_{N=0}^{\infty} \sum_{M=0}^{\infty}  c_{N+2,M} (N+2)(N+1)\nu^{N}\bar\omega^{M}  = \sum_{N=0}^{\infty}\sum_{M=2}^{\infty}c_{N+1,M-1} (N+1)(M-1) \nu^{N}\bar\omega^{M}.
	\label{diri-eq-kernel-new}
\end{equation}

From \eqref{diri-eq-kernel-new} for $N\in \mathbb N_{0}$ and $M\geq 2$, we have 
\begin{equation}
\label{789}
c_{N+2,M} (N+2) = (M-1)c_{N+1,M-1}.
\end{equation}

\

We now check that all off diagonal entries of $C(k)$ are indeed zero. Let $M=0$; then from \eqref{diri-eq-kernel} with the change of variable $N=n-2$ gives us
\[\sum_{N=0}^{\infty} c_{N+2,0} (N+2)(N+1) \nu^N=0,\]
so we have
\[ c_{N+2,0}=0 \quad{\rm for} \quad N\geq 0 .\]
Let $M=1$; then from \eqref{diri-eq-kernel-new} we get
\[ c_{N+2,1}=0\quad{\rm for}\quad N\geq 0. \]
Hence all off diagonal entries of $C(k)$ are zero. Since $k(0,0)=0$ we get that $c_{0,0}=0$. Finally
we set  $M=N+2$ in \eqref{diri-eq-kernel-new}, and get
\begin{equation}
  \label{werty}
  c_{N+2,N+2}(N+2)=(N+1)c_{N+1,N+1}, \quad N=0,1,\ldots
  \end{equation}
From \eqref{werty} we get $c_{N,N}=\frac{1}{N}$ for $N\ge 1$, and the proof is complete. \qed\mbox{}\\

{\bf Acknowledgment:}
I would like to thank the referee for his/her time and comments that helped improve the paper.

\bibliographystyle{plain}
\def\cprime{$'$} \def\cprime{$'$} \def\cprime{$'$}
  \def\lfhook#1{\setbox0=\hbox{#1}{\ooalign{\hidewidth
  \lower1.5ex\hbox{'}\hidewidth\crcr\unhbox0}}} \def\cprime{$'$}
  \def\cprime{$'$} \def\cprime{$'$} \def\cprime{$'$} \def\cprime{$'$}
  \def\cprime{$'$}

\end{document}